\documentclass[letterpaper, 10 pt, conference]{ieeeconf}  % Comment this line out if you need a4paper
\IEEEoverridecommandlockouts                             
\usepackage{cite}

\usepackage{amsmath,amssymb,amsfonts, amsthm}
\usepackage{graphicx}
\usepackage{textcomp}
\usepackage{xcolor}
\usepackage{subfigure}
\usepackage{enumerate}
\usepackage{multirow,multicol,threeparttable}
\usepackage{xspace}
\usepackage{soul}

\usepackage{algorithm}
\usepackage{algorithmicx}
\usepackage[noend]{algpseudocode}

\algnewcommand\algorithmicinput{\textbf{Input:}}
\algnewcommand\Input{\item[\algorithmicinput]}
\algnewcommand\algorithmicoutput{\textbf{Output:}}
\algnewcommand\Output{\item[\algorithmicoutput]}

\newtheorem{theorem}{Theorem}

\usepackage[hidelinks]{hyperref} % Linking numbers to their refs.
\usepackage[capitalise]{cleveref}
\crefformat{equation}{(#2#1#3)}
\crefmultiformat{equation}{(#2#1#3)}
{ and~(#2#1#3)}{, (#2#1#3)}{ and~(#2#1#3)}
\crefrangeformat{equation}{(#3#1#4) to~(#5#2#6)}
\def\BibTeX{{\rm B\kern-.05em{\sc i\kern-.025em b}\kern-.08em
    T\kern-.1667em\lower.7ex\hbox{E}\kern-.125emX}}

\newcommand{\apriori}{\emph{a priori}\xspace}

\newcommand{\revised}[1]{\textcolor{black}{#1}}

\setstcolor{red}

\title{\LARGE \bf
Path Planning and Energy Management of Hybrid Air Vehicles for Urban Air Mobility
}

\author{Satyanarayana G. Manyam$^{1}$, David W. Casbeer$^{2}$, Swaroop Darbha$^{3}$, \\ Isaac E. Weintraub$^{2}$ and Krishna Kalyanam$^{4}$% <-this % stops a space
\thanks{$^{1}$Satyanarayana G. Manyam is with the Infoscitex corporation, a DCS Company, Dayton, OH, USA 
        {\tt\small msngupta@gmail.com}}%
\thanks{$^{2}$David W. Casbeer and Isaac E. Weintraub are with the Controls Science Center of Excellence, Air Force Research Laboratory,
        WPAFB, OH, USA
        {\tt\small david.casbeer@us.af.mil, isaac.weintraub.1@us.af.mil }}%
\thanks{$^{3}$Swaroop Darbha is with Texas A \& M University, College Station, TX, USA
        {\tt\small dswaroop@tamu.edu }}%        
\thanks{$^{4}$Krishna Kalyanam is with the Aviation Systems Division, NASA Ames Research Center, Moffett Field, CA, USA
        {\tt\small krishna.m.kalyanam@nasa.gov }}
\thanks{Distribution Statement A. Approved for public release, distribution unlimited. Case Number: AFRL-2022-0781.}
}
\begin{document}

\maketitle
\thispagestyle{empty}
\pagestyle{empty}

%%%%%%%%%%%%%%%%%%%%%%%%%%%%%%%%%%%%%%%%%%%%%%%%%%%%%%%%%%%%%%%%%%%%%%%%%%%%%%%%
\begin{abstract}
A novel coupled path planning and energy management problem for a hybrid unmanned air vehicle is considered, where the hybrid vehicle is powered by a dual gas/electric system. Such an aerial robot is envisioned for use in an urban setting where noise restrictions are in place in certain zones necessitating battery only operation. \revised{We consider the discrete version of this problem, where a graph is constructed by sampling the boundaries of the restricted zones, and develop a path planning algorithm.} The planner simultaneously solves the path planing along with the energy mode switching control, under battery constraints and noise restrictions. This is a coupled problem involving discrete decision making to find the path to travel, and determining the state of charge of the battery along the path, which is a continuous variable. A sampling based algorithm to find near optimal solution to this problem is presented. To quantify the efficacy of the solution, an algorithm that computes tight lower bounds is also presented. The algorithms presented are verified using numerical simulations, and the average gap between the feasible solutions (upper bounds) and the lower bounds are, empirically, shown to be within $15\%$ of each other. 
% \textcolor{red}{KMK-within 15\% of what? each other? optimal cost?}
% \textcolor{red}{KMK-are you also providing upper bounds? you only mentioned lower bounds before. So, what is the gap? Is $15\%$ considered good enough?} SGM: near optimal solution above also mean upper bound, changed the wording for clarity
\end{abstract}

\section{Introduction} \label{sec:intro}

% There is an urgent need to develop green technologies that reduce non-renewable fuel usage and, as a consequence, reduce carbon emissions \cite{Boukoberine2019,Townsend2020}. 
In the area of urban air mobility (UAM) and drone delivery, many commercial ventures are considering electric propulsion aircraft \cite{hasan2019urban,patterson2021initial}.
%\cite{cite eVTOL and package delivery}
%
% \textcolor{red}{KMK- cite references. Also, UAM and ``package delivery" are two completely separate application areas per NASA- they are vertically separated and also the a/c are heavier for passenger payload and much lighter quad-coptors for drone delivery. I would recommend you stick to UAM for this paper. Please cite some NASA UAM and UTM papers- you will find plenty of AIAA references - will help my cause ! :)} SGM: added some references.
Given the deficiencies in state-of-the-art lithium-ion battery energy density and fuel cell technology~\cite{Button2021-ra,Smil_CarbonAlgebra}, it is prudent to consider alternative technologies that can help reduce our carbon footprint in the near term. Furthermore, as the world looks for faster modes of transportation and quicker delivery of goods, our skies will become saturated with the noise from these drones  \cite{rizzi2020urban}. In the most prevalent use cases, these vehicles will operate in locations where such droning background noise is unacceptable. A gasoline-electric hybrid aerial robotic vehicle is well suited for UAM or drone delivery applications, where the gasoline engine provides long endurance and electric motor facilitates the low noise mode \cite{fredericks2013benefits}. \revised{The perceived noise level when the aerial vehicle is powered by electric motor is considerably less compared to a gasoline engine \cite{cabell2016measured, kim2018review}.} 

% This paper investigates how hybrid technologies can be used to simultaneously reduce fossil fuel utilization and mitigate noise in specified areas.

%The applications of robotic air-vehicles have been growing, and the new frontiers of this include urban air-mobility and drone delivery. There are several energy sources that powers these air vehicles, \emph{viz.}, gasoline powered, solar, fuel-cells, electric etc. A review of the energy management and power architectures are presented in \cite{Boukoberine2019,Townsend2020}. Each of these have their strengths and weaknesses with respect to weight contributions, charging and discharging times, size, payload etc. 

%In terms of acceptable aircraft noise, a metric that is often used is the Sound Pressure Level (SPL) \cite{falck2018multidisciplinary}, and it is inversely proportional to the distance between the acoustic source and the observer. It is reasonable to assume that the SPL drops below the allowed limit given sufficient distance between the source and observer. Therefore, to keep the SPL below a specified limit in certain region, there exists an offset region in which the air-vehicle cannot operate in gasoline mode. 
% By optimally managing the energy switching modes of the aircraft, we can reduce fuel consumption while meeting noise constraints.
%  \textcolor{red}{KMK-moved the noise elements to one paragraph}
 
\revised{In this letter, we consider a noise constrained path planning problem for a hybrid gasoline-electric unmanned air vehicle. We assume the robotic vehicle is equipped with a series hybrid architecture \cite{lieh2011design}, where the propellers are powered by an electric motor that draws power from either the gasoline engine-generator or from the battery. The gasoline engine, when run at full capacity (maximum flow rate of fuel), can run the motor and also charge the battery; or the engine can be run at a limited capacity to produce only sufficient power to run the motor. However, due to frictional losses, it is more efficient to run at full capacity and charge the battery to full whenever possible, and then power the motor with the battery. Therefore, we assume that the fuel rate is at maximum whenever the gasoline mode is chosen. We further assume that the architecture facilitates \emph{instant} effortless switching between these two modes. The robotic vehicles considered can make sharp turns similar to quad-rotors, and therefore, we do not consider the kinematic constraints.}

% \textcolor{red}{KMK- why keep referring to ``robotic" vehicle? I prefer unmanned or UAV. This avoids confusion with robotic ground vehicles. I would insert a picture here of the energy flow/hybrid architecture. It is not clear that you can operate in gas-only or gas+ charging the battery.} SGM: calling it robotic vehicle makes the ras community happy

The path planning problem involves finding a path between a pre-specified start and goal locations in the presence of quiet zones. The aerial robot is allowed to pass through  quiet zones, however it must be powered by the electric mode (gasoline-engine turned off) while flying above such zones. A candidate path can be divided into several segments, where each segment could be either gasoline or electric mode. \revised{The cost we have considered in this letter is the fuel cost, and therefore the objective is to minimize the length of the segments that are traveled in gasoline mode. This cost is appropriate for commercial applications where the objective is to minimize the fuel consumption. However, the framework presented here can easily modified for any other application that aims to minimize a different objective, for example, travel time.} The decision making involves finding the path, while simultaneously determining \revised{the switching points from gas to electric and vice-versa.} To do so, the planner must determine the segments along the path where the power source is the gasoline engine or the electric motor, such that state of charge remains within the capacity limits.
% \textcolor{red}{KMK- why is the cost different? Just specify the cost, so it is clear. If it is ``travel time", then there is no difference} The cost of travel is different for each of these modes, and the objective is to minimize the total cost of the path, \emph{i.e.}, sum of cost of travel on each segment. SGM: done
% \textcolor{red}{KMK- I personally prefer calling it ``determining the switching points from gas to electric and vice-versa"} SGM: done

To find the optimal paths, one needs to model the battery characteristics, the rate of charge of the battery while the robot travels in gasoline mode and the rate of discharge in electric mode. \revised{For simplicity, in the operational limits of the battery charge level, we assume that the rates of discharge and recharge with respect to distance traveled are constants.} Let $q(s)$ be the variable representing battery charge along a path, where the parameter $s$ represents the length of the path traveled, then
\begin{align}
        &q'(s) = 
        \begin{cases}
                -\alpha &\text{if electric mode},\\
                \beta &\text{if gasoline mode}, \\
        \end{cases} \label{eq:batterymodel}\\
        &q_{\min}\le q(s) \le q_{\max} \label{eq:chargelimits},
\end{align}
where $\alpha$ and $\beta$ are the rates of discharge and recharge per unit distance traveled. This model allows for quick evaluation of feasibility for a given state of charge at initial and final position of a segment of the path. The path planning algorithm we present in this paper can accommodate higher fidelity battery models. %However, for ease of exposition, we use this simple model.

\emph{Toy Example:} Suppose the initial charge is $80\%$, and final charge level is required to be at least $50 \%$. A feasible path and the corresponding charge profile along the path are shown in Fig. \ref{fig:ex1}, where there exists one quiet zone, shown in grey. The path is demarcated and shown in magenta and green for the gasoline and electric modes, respectively. Note that the charge profiles are linear due to the model assumed in \eqref{eq:batterymodel}. 
In this example, the robot travels using the gasoline mode and switches to electric mode before it enters the quiet zone and continues in the electric mode throughout the quiet zone, and switches to the gasoline mode after exiting from the quiet zone. It is clear that the mode of the segment in the quiet zone is electric, and the other segments could be either gasoline or electric. \emph{The planning algorithm must chose the modes and switching points along a path that minimizes the total fuel cost.} 

Let us represent a path as a series of segments $\{ (v_s, v_1), (v_1, v_2), \ldots (v_n, v_t) \}$. For a given path, one solution approach would be to determine the charge levels $\{q_1, \ldots q_n \}$ at the vertices $\{v_1, \ldots v_n \}$ that are feasible with respect to the battery model, and the state of charge satisfies \eqref{eq:chargelimits} everywhere along the path. However, for the problem considered in this paper, the path itself is a decision variable; this coupling between the path planning and charge profile planning poses a challenge.

\begin{figure}
        \centering
        \vspace{8pt}
        \includegraphics[width=.75\columnwidth]{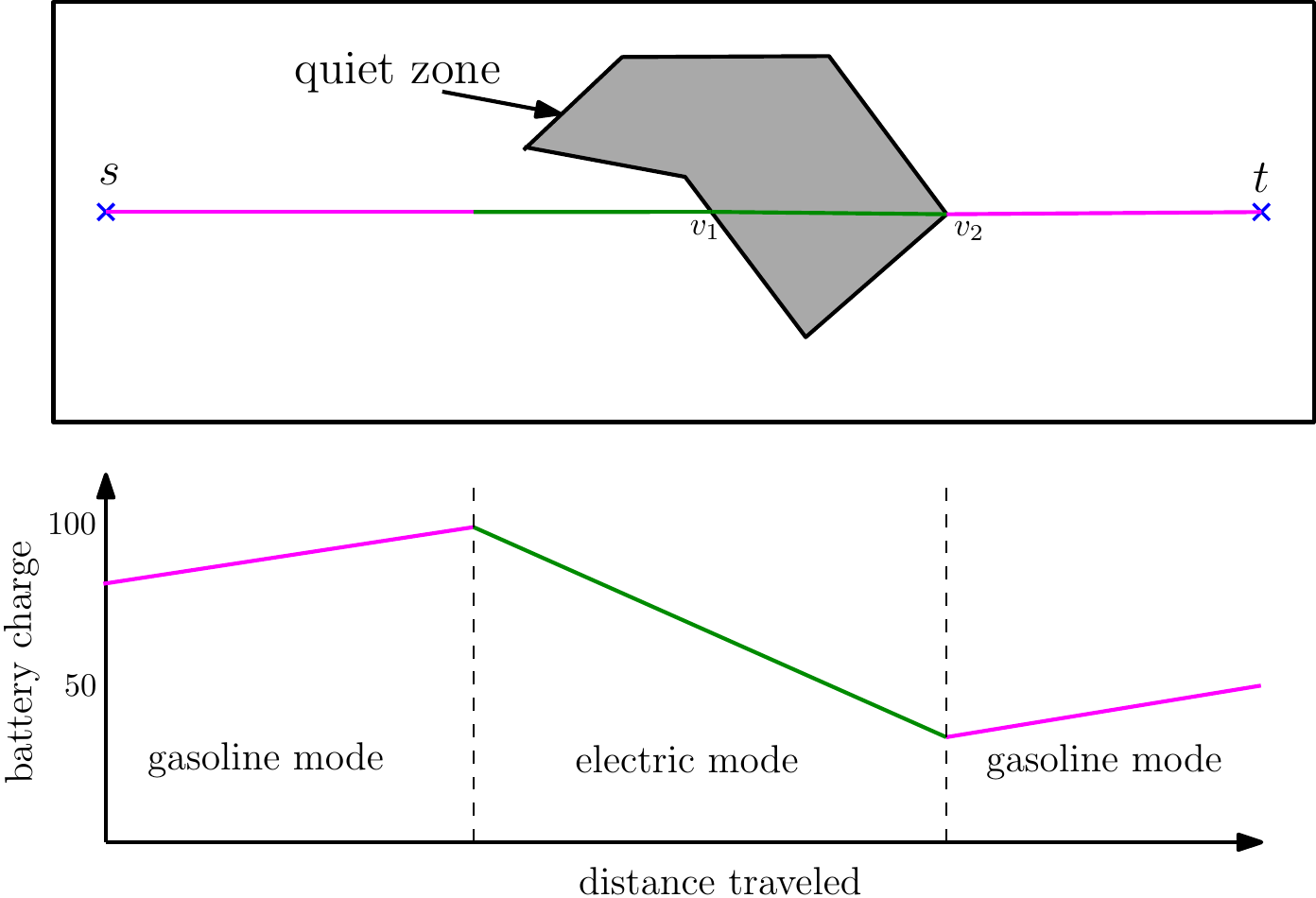} 
        \caption{An example of a feasible path for hybrid navigation}
        \label{fig:ex1}
\vspace{-15pt}
\end{figure}

On its own, path planning in the presence of \revised{quiet zones} resides in a continuous space resulting in an infinite dimensional problem. Sampling the continuous space is a popular technique used to translate the problem to a discrete planning problem. We discretize the problem by sampling the boundaries of the \revised{quiet zones}, and generate a graph similar to a visibility graph \cite{deBerg2000}, shown in Fig. \ref{fig:smplGraph}. \revised{We address the discrete version of the problem by formulating the path planning and the energy management problem on the discrete graph.} It allows us to reduce the infinite dimensional path planning problem into a finite dimensional discrete problem. There is a loss of optimality using this approach, however, it yields a simplified and potentially tractable approach to solve the coupled infinite dimensional problem. Note that, the loss of optimality depends on the sampling intervals along the boundaries, and therefore the loss could be made sufficiently small with sufficiently large sampling rate. We also develop an algorithm that produces tight lower bounds to this problem, and thereby corroborates the quality of the feasible solutions. \revised{{\it The key idea to obtain this lower bound is to partition the domain of a set of continuous variables, that determine the state of charge at each vertex, into a set of sub-intervals. The state of charge at these vertices is allowed to be discontinuous, \emph{i.e.} the vehicle can arrive at and exit a vertex with different charge levels; but the two charge levels are forced to lie in one of the sub-intervals. This problem is relaxed compared to the original because we are allowing it to violate the continuity of the state of charge at the vertices.} } 

\emph{Related literature:} \revised{The path planning problems are solved using sampling based algorithms such as RRT \cite{lavallerrt}, RRT$^*$~\cite{karaman2011sampling}, BIT$^*$~\cite{gammell2020batch}, road-map based search techniques such as PRM methods \cite{kavraki1996probabilistic}, incremental graph search methods such as D$^*$ Lite \cite{koenigDslite}. Other path planning techniques include visibility graph, Voronoi diagram based or potential field methods \cite{surveypathplan}.} There exists few results in the area of energy aware path planning. In~\cite{lin2018}, a planning problem is considered where a robot needs to accomplish a set of goals while maintaining a minimum energy threshold. In \cite{Franco2015}, an energy aware coverage planning problem is addressed, where a robot needs to cover an area that minimizes the total energy consumption. A planning problem to minimize energy consumption in the presence of disturbances is addressed in~\cite{bezzo2016}. The paper uses a model predictive approach to estimate safety critical states, and presents a self triggering schedule to re-plan, and is compared to periodic re-planning. Zhang et al. addressed a route planning for a plug-in hybrid vehicle is addresses in \cite{Zhang2020} that minimizes energy consumption. The authors addressed the coupled routing problem that aims to simultaneously optimize the decision making for path and the power management. 

The problem considered in this paper addresses the path planning problem for a hybrid aerial robot, in the presence of (noise-)restricted zones. Here, we aim to optimize the travel cost for a robotic vehicle that can switch between two different modes. To the best of our knowledge, a path planning that allows a robot to switch between travel modes, with constraints on the modes in certain regions, has not addressed before in the literature. The application of this novel problem to urban air mobility and drone delivery is of particular relevance and deserves attention. \revised{In  sampling based and incremental search techniques, a tree is iteratively constructed by following a sampling procedure. To determine the switch points along the path while simultaneously growing the tree is not possible without decoupling the path planning and energy management problems. The novelty of the proposed algorithm lies in addressing the coupled problem that involves path planning and energy management. Moreover, the presented technique facilitates computation of the lower bound to the optimal solution, which ratifies the quality of the feasible solutions produced.}

The main contributions of this work are: $(i)$ we present a novel path planning and energy management problem for a hybrid robot suited for an urban air mobility application, $(ii)$ we formulate the coupled problem on the discrete graph that involves discrete decision making for the path, and continuous variables for the state of charge along the path, $(iii)$ we present a sampling based approach to find near optimal solutions, $(iv)$ we develop a partitioning algorithm to find tight lower bounds to the optimal solution, $(v)$ we validate the presented algorithms using numerical experiments on some benchmark areas of operation.

The rest of the paper is organized as follows. In Section~\ref{sec:prblmform}, a graph is constructed by sampling the boundaries of the \revised{quiet zones}, and the hybrid path planning is defined on this graph. The algorithms to compute near optimal solutions and tight lower bounds using sampling and partitioning approach, respectively, are presented in Section~\ref{sec:approach}. The algorithms presented were tested using computational experiments, the results are presented in Section~\ref{sec:results}, and some concluding remarks are provided in Section~\ref{sec:concl}.

\section{Problem Formulation} \label{sec:prblmform}

Before formalizing the problem, we define the graph, $G_s$ using Algorithm \ref{alg:smplGraph}, which takes as inputs the start and goal locations, $v_s$ and $v_t$, a set of \revised{quiet zones}, $\mathcal{O}$, and a sampling interval, $\delta L$. \revised{The boundaries of each quiet-zone are sampled uniformly by choosing a point at every $\delta L$ units of distance. The vertex set, $V_s$, is built by creating a vertex corresponding to these samples, the start and goal positions (steps \ref{alg:Vs1}--\ref{alg:Vs2} of Algorithm \ref{alg:smplGraph}). An edge is added if the two vertices are visible\footnote{Two vertices are considered visible if the straight line connecting them does not intersect with any other quiet zones.}; note that, the edges between the vertices that belong to the same quiet zone are added too (step \ref{alg:edgFeas1}). The edge set consists of edges both inside and outside the quiet zones.} The sampled graph of the example problem is shown in Fig. \ref{fig:smplGraph}. The dotted lines inside the quiet zone represent the feasible segments of the path, where the mode of travel is restricted to electric. \revised{Though we consider only the polygonal quiet zones in the examples, this method directly applies to non-polygonal quiet zones too. For non-convex shapes, one may consider convex-hull as done in \cite{huang2017viable}, and generate the graph. If one needs to consider obstacles along with quiet zones, it is sufficient to add the vertices of the obstacles to $V_s$, and corresponding edges that does not intersect with the obstacles.}

% 

%\textcolor{red}{KMK- Again the confusion between obstacle and quiet zone. The picture says ``obstacle", which means you can't go inside but the text says quiet zone. Perhaps you mean ``obstacle to gas mode"?} SGM: It was wrong, we are not considering obstacles, changed all of them to "quiet zones". However, obstacles could also be handled with this construction 
% \textcolor{red}{KMK- Is it uniform sampling in distance? It is not clear what the units of $\delta L$ are.} SGM: yes, changed the wording
% \textcolor{red}{KMK- This is not the least bit clear. What do you mean ``visible"? Looking from what direction? Perhaps best if you add an example picture here showing one edge that gets added vs another one that does not get added to the list because it is occluded.} SGM: added footnote

\begin{figure}[htpb]
        \vspace{8pt}
        \centering
        \includegraphics[width=.90\columnwidth]{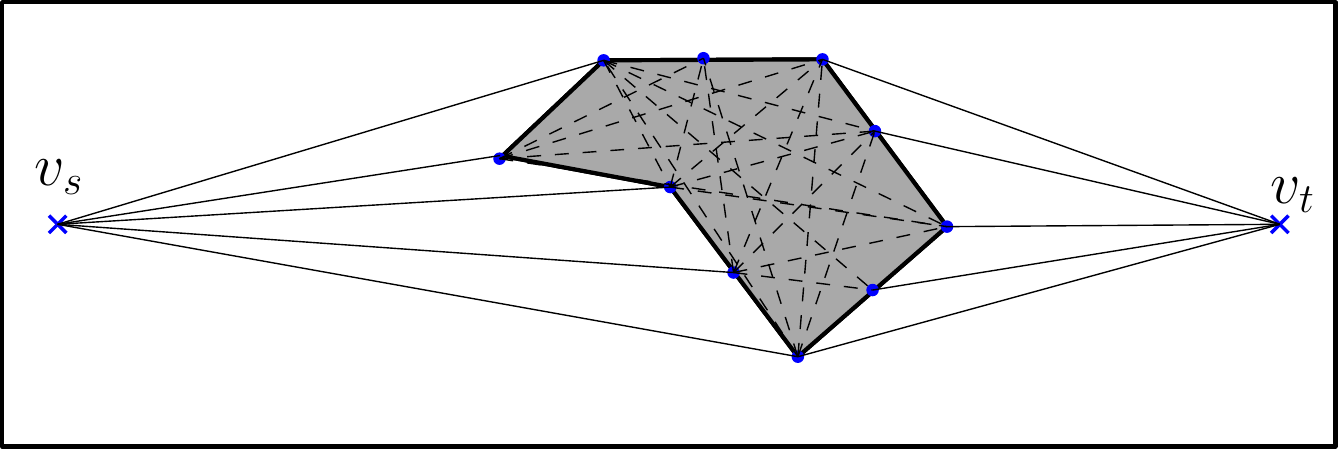} 
        \caption{A graph generated by sampling the boundaries of the \revised{quiet zones}.}
        \label{fig:smplGraph}
        \vspace{-10pt}
\end{figure}

We formulate and solve the path planning and energy management problem using the sampled graph, $G_s(V_s, E_s)$. With this approximation, the path planning reduces to finding an ordered sequence of nodes $S_p :=\{v_s, v_{s_1}, \ldots v_{s_k}, v_t \}$ on graph $G_s$ along with the state of charge of the battery at each of these nodes, $\{q_s, q_{s_1}, \ldots, q_{s_k}, q_t  \}$. Let $v_l$ indicate the $l^{th}$ node in $S_p$.  The following must hold for $S_p$ and the corresponding battery charge at each node $v_l\in S_p$: $(i)$ the state of charges, $q_{s_l}, q_{s_{l+1}}$ on every edge $(v_l,v_{l+1}) \in S_p$ satisfies the battery dynamics \eqref{eq:batterymodel} and \eqref{eq:chargelimits}, $(ii)$ the segments of the path in the quiet zones are in electric mode, \revised{($iii$) the state of charge of the battery at the terminal node is greater than a specified value, $q_{goal}$,} and $(iv)$ the total cost of travel is minimized. The cost we aim to minimize is the total fuel consumed while traveling in gasoline mode. Therefore, each edge has a zero cost if it is traveled completely in electric mode. 

\begin{algorithm}
        \caption{Construction of the sampling graph \label{alg:smplGraph} }
        \begin{algorithmic}[1]
        
                \Function{SamplingGraph}{$v_s, v_t, \mathcal{O}, \delta L$}
                \State $V_s \gets \textproc{InitiateNodeSet()}$
                \State $E_s \gets \textproc{InitiateEdgeSet()}$            
                \State $V_s \gets V_s \cup \{v_s, v_t \} $          \label{alg:Vs1}   
                \For {$O_i \in \mathcal{O}$}                        
                        \State $V_s \gets V_s \cup \textproc{DiscreteSampling}(O_i, \delta L)$    \label{alg:Vs2}                                            
                \EndFor	
                \For {$v_i, v_j \in V_s$}
                \If {$\textproc{CheckEdgeFeasibility}(v_i, v_j)$} \label{alg:edgFeas1}
                        \State $E_s \gets E_s \cup (v_i, v_j)$                                                 
                \EndIf
                \EndFor
                \State $G_s \gets \textproc{CreateGraph}(V_s, E_s)$
                \State \Return{$G_s$}
                \EndFunction
        \end{algorithmic}
\end{algorithm}

Let $I:=\{1, \ldots, |V_s| \}$ be the set of indices of the vertices in $V_s$. \revised{For any $i,j\in I$, let $x_{ij}$ denote the binary variable such that $x_{ij}=1$ if an edge $(v_i,v_j)$ is chosen to be on the path, and $x_{ij}=0$, otherwise. Given the states of charge of the battery $q_i$ and $q_j$ at the vertices $v_i$ and $v_j$, respectively, let $c_{ij}(q_i,q_j)$ be the cost of travel from $v_i$ to $v_j$.} Let $\mathbf{x}$ represent a matrix of all binary variables, and $\mathcal{X}$ be the set of all feasible paths from $v_s$ to $v_t$ in $G_s$. The optimization problem, $\mathcal{P}_1$, is stated as the following.
%
% \textcolor{red}{KMK- Are $s, v_s$ and $t,v_t$ one and the same?} SGM: v_s and v_t are vertices in V_s, s and t are positions in the map; changed them to v_s, v_t above.
%
\begin{align} \label{prblm:p1}
        \mathcal{P}_1: \min_{\mathbf{x} \in \mathcal{X}, q_k \in [q_{\min}, q_{\max}], \forall k \in I } \sum_{i,j \in I} x_{ij}c_{ij}(q_i, q_j)
\end{align}
In the following sections, we present the algorithms that simultaneously optimizes the two sets of variables, $\mathbf{x}$ and $\{ q_i, i \in V_s \}$. The approach involves sampling of the state of charge at every vertex in $V_s$, constructing a graph $G_u(V_u, E_u)$, and solving a shortest path problem on this graph; the solution to the shortest path problem on $G_u$ produces a feasible solution (upper bound to optimal solution) to $\mathcal{P}_1$. \revised{For a given map with quiet zones, the construction of the base graph without the start and goal can be done offline. We only need to add the start and goal vertices and corresponding edges to compute the shortest path, which significantly reduces the online computation time.} We also present a partitioning approach, similar to the upper bounding algorithm, that produces tight lower bounds. In this way, we provide both upper and lower bounds to the optimal solution and hence the gap between the two is the maximum gap between the optimal and feasible (upper bound) solutions.
% \textcolor{red}{KMK: there are no obstacles, right? Just call them quiet zones for consistency throughout the text. Also, obstacles immediately remind people of keep out zones, what we have is keep in if you are quiet :D}
% \textcolor{red}{KMK- now it is clear where the UB comes from - this was not clear (to me) in the abstract. Perhaps you can clarify in the abstract as well.}
%The lower bound corroborates the quality of the feasible solution and provides a limit to the maximum gap between the upper bound and the optimal solution.

\section{Technical Approach} \label{sec:approach}
The algorithms to compute the upper bounds and lower bounds follow a  method similar to Algorithm \ref{alg:smplGraph} of creating nodes and edges. These algorithms \emph{sample} and \emph{partition} the state of charge at each vertex in $V_s$. In the upper bounding algorithm, the state of charge at each vertex is uniformly sampled, creating a new node for each sample of state of charge at each vertex. In contrast, the lower bounding algorithm partitions the feasible interval of state of charge into small sub-intervals and builds a graph where each node represents a sub-interval of the state-of-charge at each vertex in $V_s$. 
% \textcolor{red}{KMK- Gupta, please give an example here of the partitioning/sampling with numbers. It will greatly benefit the reader}
This partitioning approach is similar to that found in \cite{manyam2017tightly,manyam2018tightly,rathinam2019near,vana2015dubins}. For coupled optimization problems involving both discrete and continuous variables, this approach is found to produce tight lower and upper bounds, and therefore guarantees the quality of the upper bounds with respect to the optimal solution. 
% \textcolor{red}{KMK- When you say $15\%$, you mean the cost of the upper bound and lower bound solutions are no more and no less than $1.075$ and $0.925$, respectively compared to the optimal cost of $1$. Is this correct or perhaps you mean it is not equidistant from the optimum but the gap is no more than $.015$? }
% SGM: It is the later, we don't know where the optimum lies. Since upper bound is within 15% from lower bound, it is within less then 15% from optimum.

\subsection{Feasible Solution (Upper bounds)}

In problem $\mathcal{P}_1$, at any vertex $v_k \in V_s \setminus \{v_s,v_t \}$, the charge $q_k$ is a continuous variable and $q_k \in [q_{min}, q_{max}]$. To compute near optimal feasible solutions, for every $v_k \in V_s \setminus \{v_s, v_t \}$, we chose a discrete set of values, $Q^k_u$, sampled uniformly in the interval $[q_{min}, q_{max}]$.  Such a sampling procedure transforms $\mathcal{P}_1$ into the discretized problem $\mathcal{P}_2$ below:
\begin{align}
        \mathcal{P}_2: \min_{\mathbf{x} \in \mathcal{X}, q_k \in Q^k_u \forall k \in I} \sum_{i,j \in I} x_{ij}c_{ij}(q_i, q_j).
\end{align}
By construction, any solution to $\mathcal{P}_2$ is a feasible solution to $\mathcal{P}_1$.
To solve $\mathcal{P}_2$, we construct a graph $G_u(V_u, E_u)$, where the set of nodes, $V_u$, consists of all nodes corresponding to every combination of $v_k \in V_s \setminus \{v_s,v_t \}$ and $q_k \in Q^k_u$. One can choose the discrete set $Q^k_u$ to be the same for every node $v_k$, and let it be $Q_u$. For the start node, there is only one value $(v_s,q_{init})$, and the feasible charges for the goal node is sampled from the set  $[q_{goal}, q_{max}]$; denote this set of charges as $Q_{goal}$. Then $V_u$ contains the Cartesian product of the sets $V_s \setminus \{v_s,v_t \}$ and $Q_u$, and the nodes corresponding to the start and goal nodes, that is $V_u := \{ V_s \setminus \{ v_s,v_t\} \times Q_u \} \cup \{(v_s,q_{init})\} \cup \{ \{v_t\} \times Q_{goal} \}$.  An illustration of the graphs $G_s$ and $G_u$ is shown in Fig. \ref{fig:graphsmpl} and \ref{fig:graphUB}. 

% \textcolor{red}{KMK- Gupta, can we reduce the size of the search space by only considering feasible charge sets, if you are going to use gas, the charge can only go up and if you use electric, the charge can only go down? So, if you start with a bigger set of possible charge levels, the next node you will travel to will necessarily have a smaller set? Also, it is hard to see the charge levels, intervals in the Fig. 3. As I said earlier, just give an example in the text here to help the reader understand.} 
% SGM: we do not know apriori the sequence of nodes visited, therefore, we cannot determine a priori which nodes could have smaller/larger set of charges.

For any nodes, $v_l \in V_u$, let the position corresponding to the node $v_l$ be $p_l$ and state of charge of the battery, $q_l$. We check if a feasible edge exists between a pair of nodes that complies with the battery dynamics \eqref{eq:batterymodel}, and compute the cost of such edge if it exists. For a given pair of nodes $v_{u_i}, v_{u_j} \in V_u$, where $v_{u_i} = (v_i, q_i), v_{u_j}=(v_j, q_j)$, Algorithm \ref{alg:evalEdge} checks if edge $(v_{u_i}, v_{u_j})$ is feasible with respect to the battery dynamics. The algorithm returns the cost of the corresponding edge as $zero$ if the travel mode is completely electric. \revised{Notice that the gasoline engine runs at full capacity whenever it is turned on, and therefore it is sufficient to find the length of the segments that are traveled in gasoline mode to evaluate the cost. If an edge is traveled in both gasoline and electric modes, the algorithm computes the distance traveled in gasoline mode, $\lambda d_{ij}$, and returns the fuel cost of travel, $c_f \lambda d_{ij}$, where $c_f$ is cost of fuel per unit distance traveled.}

\begin{algorithm}
        \caption{Evaluation of an edge \label{alg:evalEdge} }
        \begin{algorithmic}[1]
                \Function{SOCFeasibility}{$v_{u_i}, v_{u_j}$}
                \State $d_{ij} = |position(v_{u_i})-position(v_{u_j})|$
                % \If{$q_i = q_{max} \And q_j=q_{max}$} \label{alg:evalEdgAllGas}
                %         \State FeasCheck $\gets$ true \Comment{gasoline mode}
                %         \State $\lambda \gets 1$
                \If{ $q_j> \min (q_{max}, q_i + \beta d_{ij}) $} 
                        \State FeasCheck $\gets$ false \label{alg:evalEdgAllGas2}                               
                \ElsIf{ $ q_i - \alpha d_{ij} > q_{min} \And q_j \le q_i-\alpha d_{ij}$} \label{alg:evalEdgAllElec}
                        \State FeasCheck $\gets$ true   \Comment{electric mode}
                        \State $\lambda \gets 0$
                \Else      \Comment{partly gasoline/electric}  
                    \State $\lambda_1 \gets (q_{max}-q_i)/(\beta*d_{ij})$ \label{alg:compLambda1} 
                    \State $\lambda_2 \gets (q_{max}-q_j)/(\alpha*d_{ij})$
                    \If{$ \lambda_1, \lambda_2 \ge 0 \And \lambda_1+\lambda_2 \le 1$}
                        
                        % lamda = lambda1 + (alpha/(alpha+beta))*(1-lambda1-lambda2)
                        \State $\lambda \gets \lambda_1 + (\frac{\alpha}{\alpha+\beta})*(1-\lambda_1-\lambda_2)$ 
                    \Else
                        \State $\lambda \gets \frac{q_j-q_i+\alpha d_{ij}}{(\alpha + \beta)d_{ij}}$ \label{alg:compLambda2} 
                    \EndIf
                    \If{$0 \le \lambda \le 1$}  
                        \State FeasCheck $\gets$ true   
                    \EndIf 
                \EndIf
                \If{FeasCheck}
                \State $cost \gets c_f \lambda  d_{ij}$ \Comment{fuel cost} \label{alg:evalEdgCost}
                \EndIf
                
                \Return FeasCheck, $cost$
                \EndFunction
        \end{algorithmic}
\end{algorithm}

Algorithm \ref{alg:evalEdge} checks if ``all electric" mode is feasible  in step \ref{alg:evalEdgAllElec}. The variable $\lambda$ indicates how much of the edge is traveled in gasoline mode. When an edge is traveled in partly gasoline and partly electric mode, without loss of generality, we assume that the robot travels in gasoline mode first and then switches to electric. \revised{Further, we allow a maximum of three switch points on an edge. When an edge is sufficiently long, more than three switch points might be necessary; however, such cases can be accommodated by breaking the long edge into smaller edges by introducing artificial vertices. If an edge is traveled using both modes, the value of $\lambda$ is computed in steps \ref{alg:compLambda1}--\ref{alg:compLambda2}, and the corresponding cost is computed in step \ref{alg:evalEdgCost}. }
% \textcolor{red}{KMK- Don't say digression, the reviewer might think it is very relevant - just say we don't cover it here. If you have any explanation for why it is an easy extension, please mention it here. One simple idea might be to just break up long edges by introducing additional artificial vertices along the edge} 

To solve the problem $\mathcal{P}_2$, we construct graph $G_u(V_u, E_u)$, such that a shortest path on this graph produces a solution to the problem $\mathcal{P}_1$. The pseudocode of the algorithm that solves $\mathcal{P}_2$ is presented in Algorithm  \ref{alg:upperbound}. The problem prescribes a state of charge of the robot at the start, and therefore we can create a node, $v_{u_s}$, correspondingly, and add to $V_u$, shown in step \ref{alg:createnode_s}. In step \ref{alg:discsmpl1}, the discrete set, $Q_u$, is obtained by sampling the interval $[q_{\min}, q_{max}]$. In steps \ref{alg:createnode}--\ref{alg:createnode1}, for every combination of $(v_i, q_j), v_i \in V_s \setminus \{v_s, v_t\}, q_j \in Q_u$, we create a node and add to $V_u$. A minimum state of charge, $q_{goal}$, is required at the goal position, and to satisfy that, we sample uniformly in the interval $[q_{goal}, q_{max}]$, and create the set of nodes, $V_{goal}$ (shown in step \ref{alg:createnode2}), that correspond to the goal position and a state of charge $q_j \in Q_{goal}$.  In steps \ref{alg:createedge1}--\ref{alg:createedge4}, we construct the set of edges, $E_u$, by adding an edge for every pair of nodes in $V_u$, if Algorithm \ref{alg:evalEdge} returns a feasible solution; the costs returned by the algorithm are set as the weights of those edges. Note that there could be multiple nodes in $V_{goal}$, that correspond to the goal position, and satisfy the minimum charge required at goal. Therefore, any path from $v_{u_s}$ to a node in $V_{goal}$ is a feasible path. To find the minimum cost path, we add an another node $v_{u_t}$, that corresponds to the goal, and add zero cost edges between all nodes in $V_{goal}$ and $v_{u_t}$, shown in steps \ref{alg:createnode_t}--\ref{alg:createedges_t}. \revised{Finally, we use Dijkstra's algorithm to find the optimal shortest path from $v_{u_s}$ to $v_{u_t}$ in step \ref{alg:spGu}. The time complexity of Dijkstra's algorithm is $\mathcal{O}(|V_u|^2)$, and it returns the minimum cost path on the graph $G_u$. The trajectory for the robot is constructed using the position of the vertices in the shortest path.}

\begin{algorithm}
        \caption{Construction of the graph, $G_u$ \label{alg:upperbound} }
        \begin{algorithmic}[1]
                \Function{SamplingGraphG}{$G_s, \delta q$}
                \State $V_u \gets \textproc{InitiateNodeSet()}$
                \State $E_u \gets \textproc{InitiateEdgeSet()}$            
                \State $v_{u_s} \gets \textproc{CreateNode}(v_s, q_{init})$
                \State $v_{u_t} \gets \textproc{CreateNode}(v_t)$
                \State $V_u \gets V_u \cup \{v_{u_s}\} $    \label{alg:createnode_s} 
                \State $Q_u \gets \textproc{DiscreteSampling}([q_{min}, q_{max}], \delta q )$   \label{alg:discsmpl1}
                \State $Q_{goal} \gets \textproc{DiscreteSampling}([q_{goal}, q_{max}], \delta q )$   \label{alg:discsmpl2}

                \For {$v_i \in V_s$, $q_j \in Q_u$}       \label{alg:createnode}                 
                % \For {$q_j \in Q_u$}                        
                        \State $V_u \gets V_u \cup \textproc{CreateNode}(v_i, q_j)$   \label{alg:createnode1}                                             
                % \EndFor	
                \EndFor	 
                
                \For {$q_j \in Q_{goal}$}                        
                        \State $V_{goal} \gets V_{goal} \cup \textproc{CreateNode}(v_t, q_j)$      \label{alg:createnode2}                                          
                \EndFor	
                \State $V_u \gets V_u \cup V_{goal} $                                                
                
                \For {$v_{u_i}, v_{u_j} \in V_u$} \label{alg:createedge1}
                
                \If {$\textproc{CheckEdgeFeasibility}(v_{u_i}, v_{u_j}) $}
                        \State $feas, cost \gets \textproc{SOCFeasibility}(v_{u_i}, v_{u_j})$ \label{alg:edgFeas}
                        \If{feas}
                                \State $E_u \gets E_u \cup (v_{u_i}, v_{u_j})$  
                                \State $c_{ij} \gets$ cost  \label{alg:createedge4}
                        \EndIf
                \EndIf                
                \EndFor
                \State $V_u \gets V_u \cup \{ v_{u_t}\} $      \label{alg:createnode_t}                                           

                \For {$v_{u_k} \in V_{goal}$}
                        \State $E_u \gets E_u \cup (v_{u_k}, v_{u_t})$  \label{alg:createedges_t}   
                \EndFor

                \State $G_u \gets \textproc{CreateGraph}(V_u, E_u)$
                \State $path_u \gets \textproc{ShortestPath}(G_u, v_{u_s}, v_{u_t})$ \label{alg:spGu}
                \State \Return $path_u$
                \EndFunction
        \end{algorithmic}
\end{algorithm}

The algorithm returns the shortest path, $path_u$, which is comprised of a series of nodes $\{v_{s_1}, \ldots v_{s_p}\}$, and each of these nodes correspond to the vertices in $V_u$. This series of vertices in $V_u$ is the path taken by the robot. The states of charge corresponding to each of the nodes in $path_u$ are the charges at corresponding vertices in $V_u$. The feasibility check in step \ref{alg:edgFeas} of Algorithm \ref{alg:upperbound} ensures the feasibility of the whole path. Since, the weight of each edge is the cost of travel in gasoline mode, the shortest path is the path of minimum fuel cost. The optimal solution of $\mathcal{P}_2$ is a feasible solution to $\mathcal{P}_1$, but may not be optimal due to the discrete sampling. A tight lower bound could corroborate the quality of a feasible solution, and an algorithm to compute tight lower bounds is presented in the next section.
\begin{figure}[htpb]
        \centering 
        \subfigure[Graph constructed by sampling the boundaries of the \revised{quiet zones}]{\includegraphics[width=3in]{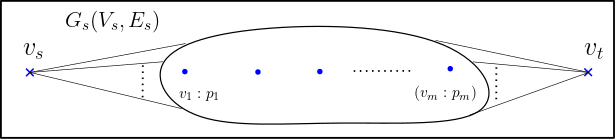} \label{fig:graphsmpl}} \\         
        \subfigure[$G_u(V_u, E_u)$ obtained by sampling charge at every $v_i \in V_s$]{\includegraphics[width=3in]{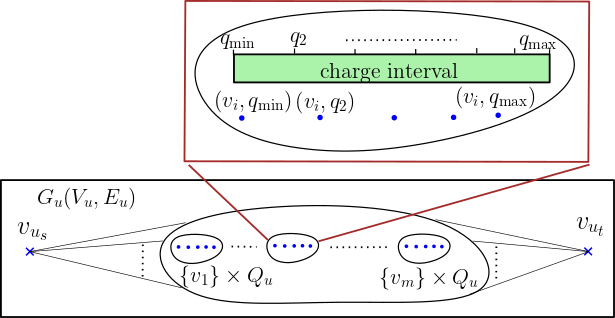} \label{fig:graphUB}} \\        
        \subfigure[Graph $G_l(V_l, E_l)$ obtained by partitioning charge into sub-intervals]{\includegraphics[width=3in]{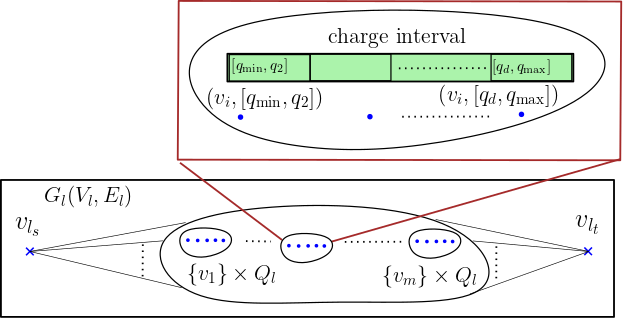} \label{fig:graphLB}}
        \caption{Graph construction to compute the upper bounds and lower bounds to $\mathcal{P}_1$}
        \vspace{-15pt}
\end{figure}
% \textcolor{red}{KMK- Question: Given the cost structure, why would the optimal solution result in a terminal charge higher than $q_{goal}$? Shouldn't it always try to use electric mode? If this were true, then we can try to solve the problem using backward DP. What are all the edges that lead to the goal node such that the charge at the goal exactly equals the required minimum? For each edge, depending on its length, the starting charge at the other end will be $q_{goal}+x$, where $x$ is the amount required to traverse the edge. If the edge is too long, then we will have a gasoline mode first that brings the charge up to the exact amount needed to bring it down to $q_{goal}$ at the end. If the edge also partly covers a quiet zone, then we simply add the charge needed to traverse that portion to the front end. I don't know- can we extend this logic and solve the problem?? Just thinking out loud. Maybe I am missing something or the suggested approach will have the same computational complexity as what you are proposing!}
% SGM: When the minimum charge required at terminal is very low, then the charge at terminal might be higher than $q_{goal}$. Consider an example scenarios where you start with 100% charge, and traveling the whole path on electric mode would bring down the charge to 50%, and if the terminal min charge is, say, 20%. It is probably possible to solve using DP kind of approach, we did not think about that much. The current approach gives lower bound, so we jumped on it.

\subsection{Lower Bounds}
To compute lower bounds, it is a common practice to relax a set of constraints and the optimal solution of the resulting relaxed problem gives a lower bound to the original problem. The Held-Karp bounds for traveling salesman problem \cite{held1970traveling},  is a good example of this technique. In recent work, for coupled problems involving discrete and continuous decision variables, a technique was developed where a set of constraints are relaxed, and another set of `loose' constraints are added, such that it produces tight lower bounds. This was shown to produce very tight lower bounds for routing problem with turn radius constraints \cite{rathinam2019near, manyam2018tightly} and for neighborhood traveling salesman problem \cite{vana2015dubins}. The problem in this paper is also a coupled problem; and we develop an algorithm using a similar idea of partitioning the continuous decision variables, that produces tight lower bounds.
% The technique involves partitioning the feasible interval of the continuous variables, and posing the resulting problem as a discrete optimization problem. 

To compute the lower bound, we pose a relaxation of the problem $\mathcal{P}_1$. Any feasible solution to $\mathcal{P}_1$ would be given as a sequence of nodes $\{v_1, \ldots v_p \}$, and a state of charge at each of these nodes, $\{q_1, \ldots q_p \}$. The continuity of the state of charge dictates that at any intermediate node $v_j$ along the path, the state of charge at the end of prior edge is same as the state of charge at the beginning of the following edge. For example, if a feasible solution contains two successive edges, $(v_i, v_j)$ and $(v_j, v_k)$, let $q^e_{ij}$ be the charge at the end of edge $(v_i, v_j)$, and $q^s_{jk}$ be the charge at the start of edge $v_j v_k$. \revised{The position of the end of the edge  $(v_i, v_j)$ is same as the start of the edge $(v_j, v_k)$; therefore, the continuity of the charge profile dictates that  $q^e_{ij}=q^s_{jk}$.} We relax this continuity constraint and allow $q^e_{ij}$ and $q^s_{jk}$ to be different but we restrict them to lie in an interval, i.e., $q^e_{ij}, q^s_{jk} \in (q_p, q_{p+1})$. 
% \textcolor{red}{KMK- Do you still impose that $q^e_{ij}\geq q^s_{jk}$?}
We refer to this relaxed problem as $\mathcal{P}_3$. Since, this is a relaxation to $\mathcal{P}_1$, every feasible solution to $\mathcal{P}_1$ is also feasible to the relaxed problem $\mathcal{P}_3$. Therefore, the optimal solution of $\mathcal{P}_3$ is a lower bound to the optimal solution of $\mathcal{P}_1$. 

To this end, at every node, we partition the feasible \emph{interval} of state of charge into $n_l$ \emph{sub-intervals}. At the vertices $v_i \in V_s \setminus \{v_s,v_t \}$, the set of intervals would be $Q_l^i=\{ [q_{\min}, q_1], [q_1, q_2 ], \ldots [q_{n_l-1}, q_{\max} ] \}$. At the goal node, the minimum charge required is $q_{goal}$, and therefore the intervals would be $\{ [q_{goal}, q_1], \ldots [q_{n_g-1}, q_{\max} ] \}$, for some $n_g$. Let $\bar{q}_i$ represent an interval of states of charge at a node $v_i$. Now the relaxed problem $\mathcal{P}_3$ is stated as follows:
\begin{align}
        \mathcal{P}_3: \min_{\mathbf{x} \in \mathcal{X}, \bar{q}_k \in Q^k_l, \forall k \in I} \sum_{i,j \in I} x_{ij}c_{ij}(\bar{q}_i, \bar{q}_j).
\end{align}
We solve $\mathcal{P}_3$ by constructing a graph, $G_l(V_l, E_l)$, similar to the $G_u(V_u, E_u)$. The nodes in $V_l$ are the combination of the vertices $v_i \in V_s$ and the intervals $\bar{q}_j \in Q^j_l$. For a pair of nodes $v_{l_i}, v_{l_j} \in V_l$, let $(q^i_k, q^i_{k+1})$ and $(q^j_m, q^j_{m+1})$ be the corresponding charge intervals. The cost of the corresponding edge in $E_l$ is the minimum cost of the edge,
\begin{equation}
        \min_{q_i \in (q^i_k, q^i_{k+1}), q_j \in (q^j_m, q^j_{m+1})} c_{ij}(q_i, q_j).   \label{eq:edgCostLB}
\end{equation}
Due to the linear battery dynamics, this cost could be found by using Algorithm \ref{alg:evalEdge} with the upper limit of the interval at the first node, and lower limit of the interval at the second node, $q^i_{k+1}, q^j_m$. The graph construction of $G_l(V_l, E_l)$ is similar to the one presented in Algorithm \ref{alg:upperbound}, however it differs only in two aspects: $(i)$ in steps \ref{alg:discsmpl1}--\ref{alg:discsmpl2}, the discrete sampling is replaced with continuous partition of the interval $[q_{min}, q_{max}]$ and $[q_{goal}, q_{max}]$, respectively, and $(ii)$ the edge cost, in step \ref{alg:edgFeas}, are assigned using the solution of \eqref{eq:edgCostLB}. Similar to the steps \ref{alg:createnode1}, \ref{alg:createnode2} of Algorithm \ref{alg:upperbound}, nodes are created corresponding to a combination of vertices $v_i \in V_s$ and intervals $\bar{q}_j \in Q^i_l$. The rest of the graph construction goes similar, and therefore, to avoid the repetition, we do not present the pseudocode for the lower bounding algorithm. An illustration of the construction of the graphs $G_u$ and $G_l$ is shown in Figs. \ref{fig:graphUB} - \ref{fig:graphLB}. In the following theorem, we formally prove that the optimal solution to $\mathcal{P}_3$ is a lower bound to the optimal solution of $\mathcal{P}_1$.

\begin{theorem}
The optimal solution of $\mathcal{P}_3$ is a lower bound to the optimal solution of $\mathcal{P}_1$.
\end{theorem}
\begin{proof}
        It is sufficient to show that every feasible solution to $\mathcal{P}_1$ is also a feasible solution to $\mathcal{P}_3$, and the cost of the solution to $\mathcal{P}_3$ is less than or equal to the cost of the solution to $\mathcal{P}_1$.
        Let $V_{feas}: = \{v_s, v_{s_1}, \ldots v_{s_p}, v_t \}$ be the sequence of vertices in a feasible solution to $\mathcal{P}_1$, and $Q_{feas} := \{q_s, q_{s_1}, \ldots q_{s_p}, q_t\}$ be the corresponding states of charge at those vertices. For each vertex $v_k \in V_{feas}$ and the corresponding state of charge $q_k$, there exists an interval $\bar{q}_k \in Q^{k}_l$, such that $q_k \in \bar{q}_k$. Construct a path in $G_l$ by identifying the nodes corresponding to $v_k$ and $\bar{q}_k$. This gives a feasible path from $v_{l_s}$ to $v_{l_t}$ in $G_l$, and let $V^l_{feas}$ be the sequence of nodes. The cost of each edge between successive nodes in $V^l_{feas}$ is less than or equal to the cost of corresponding edge in $V_{feas}$ due to \eqref{eq:edgCostLB}, and thus cost of the path in $V^l_{feas}$ is less than or equal to the one in $V_{feas}$.
\end{proof}

\section{Computational Results} \label{sec:results}

To evaluate their performance, we have tested the algorithms to compute the upper bounds and the lower bounds using several scenarios constructed from randomly generated maps and benchmark maps. \revised{Since the problem of \emph{path planning in the presence of obstacles} is closely related to the problem we consider, we use previously established benchmark maps~\cite{sturtevant2012benchmarks} to test our methods.} 
% \textcolor{red}{KMK- provide appropriate citation for this problem/benchmark. This is such a broad topic with so many variants - it would be hard to pinpoint which one you are referring to!}
% \textcolor{red}{KMK- benchmark for what problem? Given that this problem space is itself new, I was not aware of any benchmark map/areas to test our algorithms on. Perhaps you are talking about benchmark maps from a different related problem? If so, state what it is.}
We constructed the maps by randomly generating polygonal restricted zones, where the centers of the polygons are sampled from an uniform distribution. The number of sides are also randomly generated for each restricted zone. There are $10$, $15$, $20$ and $25$ restricted zones in $map1$, $map2$, $map3$ and $map4$, respectively. We have constructed two more maps using the benchmark instances `boston2' and `newyork0' from \cite{sturtevant2012benchmarks}. These are based on real world maps of a region in the cities Boston and New York; we identified the regions directly above the buildings as restricted zones, and they are appropriate for drone delivery applications as discussed in Section \ref{sec:intro}. An interested reader can access the Julia code to extract the buildings from the OpenStreetMaps data from \href{https://github.com/manyamgupta/HybridPathPlanning.git}{https://github.com/manyamgupta/HybridPathPlanning.git}. 
% These maps are shown in Figs. \ref{fig:boston} and \ref{fig:newyork}, where the restricted zones are shown in black. 
% \textcolor{red}{KMK - Drone delivery is a UTM application which is $< 400 ft$ and UAM is in a different higher altitude zone - this is the NASA/FAA construct for vertical separation of package delivery and air taxi services! I thought we were focusing on ``air taxis" and not package delivery for this paper.}

We have generated $50$ scenarios with each of the above maps, where the start and goal positions are chosen from a random distribution such that the straight line distance between them is greater than a specified limit. Further, the instances are run with different levels of discretization of the states of charge. The rate of discharge, $\alpha$, and the rate of recharge, $\beta$ are chosen such that the ratio $\frac{\alpha}{\beta}$ is equal to two, \textit{i.e.}, the battery discharges twice as fast as the it recharges per unit distance traveled.

% \begin{figure}
%         \centering
%         \subfigure[Benchmark map: boston2]{\includegraphics[width=.70\columnwidth]{obs_boston2.pdf} \label{fig:boston}} \\
%         \subfigure[Benchmark map: newyork0]{\includegraphics[width=.70\columnwidth]{obs_newyork0.pdf} \label{fig:newyork}} \\
%         \caption{Buildings extracted and identified as restricted zones from the Benchmark maps}
% \end{figure}

In Fig. \ref{fig:pathUB}, a path of the feasible solution is shown for a scenario generated in the `newyork0' map, and the charge profiles of the feasible path and the lower bound are shown in Fig. \ref{fig:chrgprfUB}. \revised{We consider no-fly zones in this scenario, shown in red, and were addressed as explained in Section \ref{sec:prblmform}}. The vertical lines are the positions of the vertices in Fig. \ref{fig:chrgprfUB}, and one may observe the charge profile is not continuous for the lower bound path. This is expected due to the relaxation of the charge continuity, and the resulting solution is a lower bound, rather than a feasible solution. For this scenario, the cost of the feasible path produced by Algorithm \ref{alg:upperbound} is $5318$ and the lower bound is given as $5137$, therefore, the gap between upper bound and lower bound is around $3.5 \%$. This infers that the feasible solution is within less than $3.5 \%$ from the optimal solution.

\begin{figure}
        \centering
        \subfigure[Solution produced by Algorithm \ref{alg:upperbound}]{\includegraphics[width=.95\columnwidth, height=150pt]{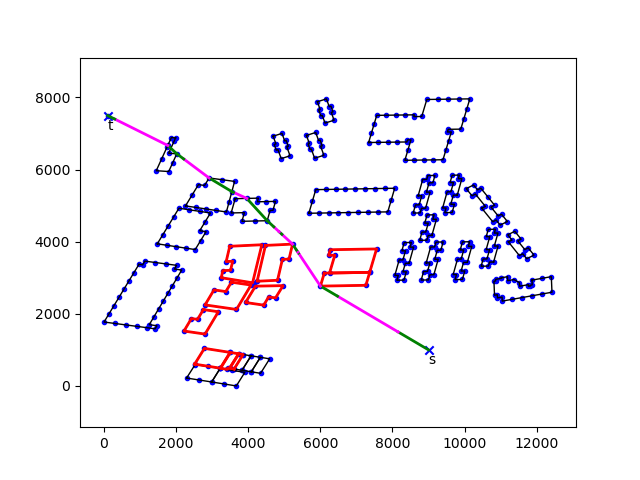} \label{fig:pathUB}} \\
        \vspace{-10pt}
        \subfigure[Charge profile of the path]{\includegraphics[width=.95\columnwidth, height=100pt]{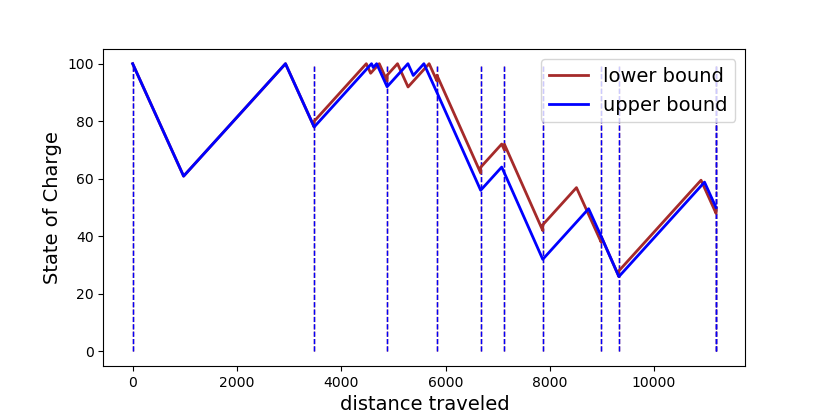} \label{fig:chrgprfUB}} \\
        % \subfigure[Charge profile of the lower bound path ]{\includegraphics[width=.90\columnwidth]{chrgProfLB_ny_scene3.png} \label{fig:chrgprfLB}}
        \caption{Results of a scenario generated in the `newyork0' map with $35$
        \revised{quiet zones}}
        \vspace{-15pt}
\end{figure}

The percent gap between the lower bounds and upper bounds is a measure of the quality of the feasible solutions. This is the maximum gap between the feasible solution and the optimal solution. For each of the maps, there are $50$ scenarios, and each scenario is solved with $20$, $30$ and $40$ discretizations. The box plot of the percentage gap is shown in Fig. \ref{fig:gaps}. Clearly, with higher sampling rate, the algorithm produces better solutions as evident from the reducing gap with higher discretizations. \revised{The maps $Boston$ and $New~York$ have a few quiet zones with very small edges, and because of these, the lower bound graph consists of many zero cost edges. For example, let $[q^a, q^b]$ be the charge interval corresponding to the vertices $v_k$ and $v_l$ in $G_l$. If the charge required to travel on the edge $(v_k, v_l)$ is less then $q^b - q^a$, this edge can be on a path with zero cost, and have same charge at the start and end of the edge. This resulted in loose lower bounds, and hence the higher gap. However, the algorithm produces tighter lower bounds by choosing sufficiently large number of sub-intervals $n_l$ while constructing the graph $G_l$.  But, this comes at a higher cost in computational time required. The computation time required to find the upper bounds and lower bounds are shown as box and whisker plots in Figs. \ref{fig:comptimeUB} and \ref{fig:comptimeLB}. The higher computation times for the $Boston$ and $New~York$ maps is due to the higher number of quiet zones.}

A significant part of the computational effort in Algorithm \ref{alg:upperbound} is spent constructing the graph $G_u(V_u, E_u)$. However, for UAM applications the restricted zones, start and end locations are know \apriori and can be computed ahead of time. For other applications, like package delivery, the restricted zones and the problem parameters are known \apriori, but the positions of the start and goal may not be known. In practice, one may construct the parts of the graph $G_u(V_u, E_u)$ offline without the nodes corresponding to the start and goal position, and edges incident on them. And when the robot's  start and goal positions are specified, the corresponding edges of the graph could be constructed and added to the graph with Algorithm \ref{alg:upperbound}. 
%\textcolor{red}{KMK - For UAM application, start and end nodes will also known a priori as the vertiport locations - these will be pre-determined as part of the infrastructure for ``air taxis" based on whole host of other considerations}
Therefore, to validate the feasibility of online implementation of this algorithm, it is sufficient to analyse the computational effort of the online part. The online computation time required by Algorithm \ref{alg:upperbound} is shown in Fig \ref{fig:comptimeUB}. Though the effort required increases with higher number of restricted zones and higher sampling rate, it is still in the order of seconds, and therefore, is viable for on-board implementation.

\revised{To evaluate the cost savings from the proposed framework, we solved the related but different path planning problem where the quiet zones are considered to be no-fly zones (i.e., feasible paths must completely avoid the quiet zones). This is done by removing the quiet zone edges from the graph $G_u$, and solving for the shortest path thereafter. Note that, we still consider the hybrid mode of the robotic vehicle, and the switching between the gasoline and electric mode still exists. If we were to restrict this to gasoline mode only, the savings would be much larger. The percentage reduction in cost using Algorithm \ref{alg:upperbound} is presented in Fig. \ref{fig:costSvng}. The cost savings are around $5$ to $10$\% for most cases, with some cases having much higher savings. This large variance is due to the randomly generated start and goal locations; the cost reduction depends on the difference in length between the shortest path that avoids the ``no-fly zones" and path generated from our work that passes through the restricted zones.}
% \textcolor{red}{KMK - So, in this case, the vehicle will still operated in electric mode whenever possible outside the quiet zones to save fuel, correct? That is why switching is still an option- perhaps you can enumerate some more on this subtle point. Also, I don't get the ``straight line" path comparison- what if the straight line between the start and end points goes through a no-fly zone?}

\begin{figure}
        \centering
        \vspace{-5pt}
        \subfigure[Average of percent gap between upper and lower bounds]{\includegraphics[width=.85\columnwidth]{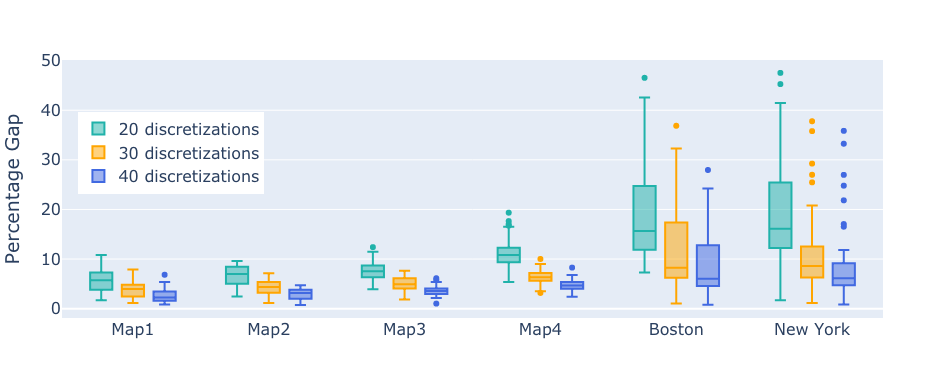} \label{fig:gaps}} \\
        \vspace{-10pt}
        \subfigure[Online computation time required for upper bound]{\includegraphics[width=.85\columnwidth]{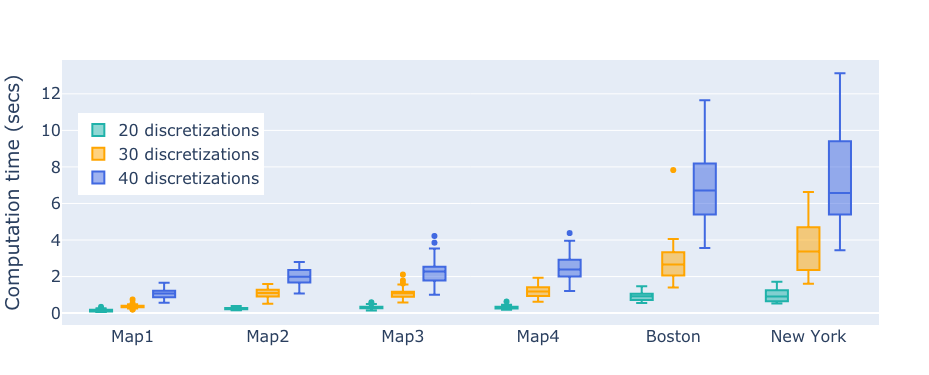} \label{fig:comptimeUB}} \\
        \vspace{-10pt}
        \subfigure[Computation time required for lower bound]{\includegraphics[width=.85\columnwidth, clip, trim={0 0 0 10}]{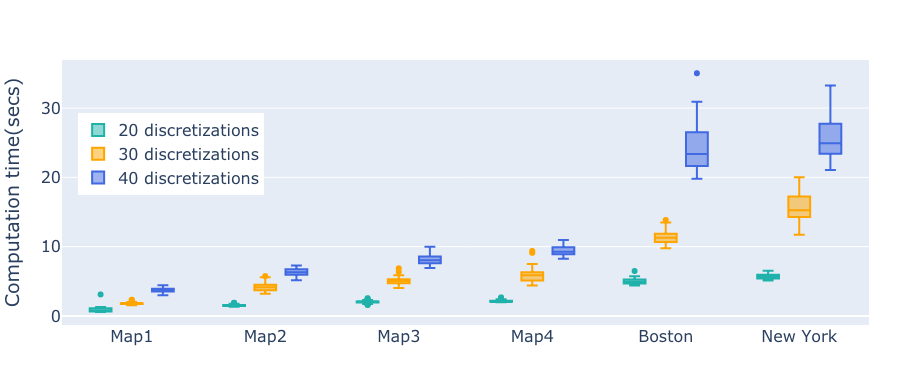} \label{fig:comptimeLB}}
        \caption{Computational results from benchmark and random maps}
\end{figure}

\begin{figure}
        \centering
        \vspace{-10pt}
        \includegraphics[width=.85\columnwidth]{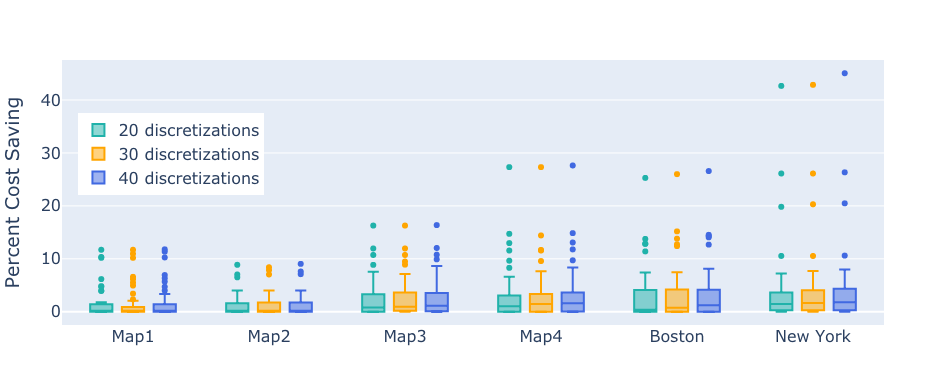} \label{fig:costSvng}
        \vspace{-10pt}
        \caption{Percentage cost reduction compared to path planning that avoids quiet zones}
        \vspace{-15pt}
\end{figure}
\vspace{-10pt}
\section{Conclusions} \label{sec:concl}
A novel hybrid path planning problem that arises from urban air mobility is presented. In this path planning problem, the hybrid vehicle is required to run in electric mode in certain regions to comply with noise restrictions. The path planner needs to generate a path and schedule for switching between gasoline and electric modes.  Algorithms based on sampling and partitioning are presented yielding upper and lower bounds to the coupled path planning and energy management problem. The paper assumes a linear battery model, but a higher fidelity model could be easily integrated with the algorithms presented here. The solutions produced by the presented algorithms are empirically shown to yield upper and lower bounds that are within $15\%$ of one another, indicating that the feasible solutions are of high quality. 

% Another significant contribution is the applicability of the partitioning algorithm to produce tight lower bounds to the coupled optimization problem.
% Moreover, the computational results show that one may get the gap between the upper and lower bounds within a desired limit with sufficiently large sampling rate, however, this comes at a higher computational cost. 

As a future research direction, one may develop an iterative scheme to compute lower bounds that refines the partitioning in each iteration only where it is necessary, and thus overcome the cost of higher sampling. Another direction of future research includes the adaptive sampling of the boundaries of the \revised{quiet zones} that chooses higher number of samples on the boundaries that are more likely to be on the path.

\vspace{-10pt}

\bibliographystyle{IEEEtran}
\bibliography{references}

\end{document}